\documentclass{article}

\usepackage{amsmath,amssymb,amsfonts,amsthm}

\newtheorem{theorem}{Theorem}[section]
\newtheorem{hypothesis}[theorem]{Hypothesis}
\newtheorem{proposition}[theorem]{Proposition}
\newtheorem{lemma}[theorem]{Lemma}
\numberwithin{equation}{section}

\theoremstyle{remark}
\newtheorem{remark}[theorem]{Remark}
\theoremstyle{remark}
\newtheorem{example}[theorem]{Example}

\newcommand{\Ric}{\mathop{\mathrm{Ric}}\nolimits}
\newcommand{\lin}{\mathop{\mathrm{lin}}\nolimits}
\newcommand{\orb}{\mathop{\mathrm{orb}}\nolimits}
\newcommand{\II}{\mathop{\mathrm{II}}\nolimits}
\newcommand{\Ad}{\mathop{\mathrm{Ad}}\nolimits}
\newcommand{\pr}{\mathrm{pr}}

\author{Artem Pulemotov\thanks{School of Mathematics and Physics, The
University of Queensland, St Lucia,~QLD 4072,
Australia}~\thanks{Department of Mathematics, The University of
Chicago, 5734 South University Ave, Chicago,~IL
60637-1514, USA}~\thanks{The author is the recipient of an Australian Research Council Discovery Early-Career Researcher Award DE150101548.} \\
\small{\texttt{a.pulemotov@uq.edu.au}}}

\title{The Dirichlet problem for the prescribed Ricci curvature equation on cohomogeneity one manifolds}

\begin{document}

\maketitle

\begin{abstract}
Let $M$ be a domain enclosed between two principal orbits on a
cohomogeneity one manifold $M_1$. Suppose $T$ and $R$ are symmetric
invariant (0,2)-tensor fields on $M$ and $\partial
M$, respectively. The paper studies the prescribed Ricci curvature
equation $\Ric(G)=T$ for a Riemannian metric $G$
on~$M$ subject to the boundary condition $G_{\partial M}=R$
(the notation $G_{\partial M}$ here stands for the metric
induced by $G$ on $\partial M$). Imposing a standard
assumption on $M_1$, we describe a set
of requirements on $T$ and $R$ that guarantee global and local solvability.
\\~\\
{\bf Keywords:} Ricci curvature, Dirichlet problem, cohomogeneity
one.
\\~\\
{\bf 2010 Mathematics Subject Classification:} 53B20, 53C20, 58J32.
\end{abstract}

\section{Introduction}

Suppose $M$ is a smooth manifold of dimension~3 or higher (possibly with
boundary) and $T$ is a symmetric $(0,2)$-tensor field on $M$. The present
paper investigates the prescribed
Ricci curvature equation
\begin{align}\label{intro_PRC}
\Ric(G)=T,
\end{align}
where the unknown $G$ is a Riemannian metric on $M$. Mathematicians have been studying~\eqref{intro_PRC}
since at least the early 1980's. We invite the reader to
see~\cite{AB87,TA98} for the history of the subject. The list of
recent references not mentioned in~\cite{AB87,TA98} includes but is
not limited to~\cite{ED02,PhD03,YR08a,YR08b,RPKT09,AP15}.

The solvability of boundary-value problems for
equation~\eqref{intro_PRC} is, by and large, an unexplored topic.
The author of the present paper made progress on this topic
in~\cite{AP11}. The main theorems of~\cite{AP11} concern the local 
solvability of Dirichlet- and Neumann-type problems
for~\eqref{intro_PRC} (i.e., solvability in a neighbourhood of a boundary point on~$M$).

It is worth noting that D.~DeTurck's study of~\eqref{intro_PRC}
underlay his discovery of the DeTurck trick. In a similar fashion,
new knowledge about boundary-value problems
for~\eqref{intro_PRC} may help answer questions about the existence and uniqueness of solutions to boundary-value
problems for the Ricci flow and the Einstein equation. Such
questions were investigated
in~\cite{YS96,MA08,JC09,MBXCAP10,MA12,PG12,AP10,AP14} and other works. A
large number still remain open.

Let $M_1$ be a smooth connected
manifold of dimension~3 or higher with $\partial M_1=\emptyset$. Consider a compact Lie
group $\mathcal G$ acting on~$M_1$. Suppose the orbit space $M_1/\mathcal G$ is
one-dimensional. It is then customary to call $M_1$ a cohomogeneity
one manifold. Such manifolds enjoy numerous applications in geometry
and mathematical physics; see, e.g.,~\cite{LBB82,CB98,ADMW99,ADMW11} and the references of~\cite{CH10}. In what follows, we suppose $M$ is
the closure of a domain on $M_1$ contained between two principal
$\mathcal G$-orbits. The boundary of $M$ is then
equal to the union of these orbits. It will be convenient for us to assume that the tensor field $T$ introduced above is defined on all of~$M_1$, not just $M$. The purpose of the present paper is to study the global and local existence of solutions to a Dirichlet-type problem for
equation~\eqref{intro_PRC} on $M$. Our main results are stated as Theorems~\ref{thm_main},~\ref{thm_loc_ex_int} and~\ref{thm_loc_ex_bdy}.

\section{Main results}\label{sec_main_results}

Because the space $M_1/\mathcal G$ is
one-dimensional, it must be homeomorphic to the real line, the closed
interval, the half-line, or the circle. In the first and the fourth case, there are no singular orbits on $M_1$. For the sake of convenience, we will assume that
$M_1/\mathcal G$ is homeomorphic to $\mathbb R$. It is easy to
state analogues of our theorems in the situations where this assumption does not hold. Pick a point in $M_1$ and denote by $\mathcal H$ the isotropy group of this point. We will use the symbol $\mathfrak g$ for the Lie algebra of $\mathcal G$. Choose an
$\Ad(\mathcal G)$-invariant scalar product $Q$ on $\mathfrak g$. Suppose $\mathfrak p$ is the orthogonal complement of the Lie algebra of $\mathcal H$ in $\mathfrak g$ with respect to~$Q$. We standardly identify $\mathfrak p$ with the tangent space
of $\mathcal G/\mathcal H$ at $\mathcal H$. The isotropy representation of $\mathcal G/\mathcal H$ then yields
the structure of an $\mathcal H$-module on $\mathfrak p$. The following
requirement will be imposed throughout the rest of the paper.

\begin{hypothesis}\label{assum_decomp_p}
The $\mathcal H$-module $\mathfrak p$ appears as an orthogonal sum
\begin{align}\label{p_decomp}
\mathfrak p=\mathfrak p_1\oplus\cdots\oplus\mathfrak p_n
\end{align}
of pairwise non-isomorphic irreducible $\mathcal H$-modules $\mathfrak
p_1,\ldots,\mathfrak p_n$.
\end{hypothesis}

Roughly speaking, Hypothesis~\ref{assum_decomp_p} ensures that $\mathcal G$-invariant (0,2)-tensor fields on $\mathcal G/\mathcal H$ are diagonal; cf.~\eqref{T_form2} and~\eqref{R_form} below. This hypothesis is rather standard. It has come up in
several papers including~\cite{ADMW00,ADMW11}.

Let the tensor field $T$ be $\mathcal G$-invariant. Assume that it is possible to construct a diffeomorphism
\begin{align}\label{dif_def}
\Psi:\mathbb R\times(\mathcal G/\mathcal H)\to M_1
\end{align}
such that the map $\Psi(r,\cdot)$ is $\mathcal G$-equivariant for every
$r\in\mathbb R$ and
the equality
\begin{align}\label{T_form1}
\Psi^*T=dr\otimes dr+T_r,\qquad r\in\mathbb R,
\end{align}
holds true. Here, $T_r$ is a $(0,2)$-tensor field on
$\mathcal G/\mathcal H$ defined
for each $r\in\mathbb R$. It is fully determined by how it acts on
$\mathfrak p$. In view of Hypothesis~\ref{assum_decomp_p}, there exist smooth functions
$\phi_1,\ldots,\phi_n$ from $\mathbb R$ to $\mathbb R$ such
that
\begin{align}\label{T_form2}
T_r(X,Y)=\phi_1(r)\,Q\big(\pr_{\mathfrak
p_1}X,\pr_{\mathfrak
p_1}Y\big)+\cdots+\phi_n(r)\,Q\big(\pr_{\mathfrak
p_n}X,\pr_{\mathfrak p_n}Y\big),\qquad X,Y\in\mathfrak p.
\end{align}
The notation $\pr_{\mathfrak p_k}X$ and $\pr_{\mathfrak p_k}Y$
refers to the projections of $X$ and $Y$ onto $\mathfrak
p_k$ for $k=1,\ldots,n$.

If the tensor field $T$ is positive-definite, it is always possible to construct the diffeomorphism $\Psi$. Indeed, in this case, we can interpret $T$ as a Riemannian metric on $M_1$ and consider a unit speed geodesic $\eta$ with respect to this metric. Assuming $\eta$ is orthogonal to all the principal orbits, we define $\Psi(r,g\mathcal H)=g\eta(r)$ for all $r\in\mathbb R$ and $g\in\mathcal G$. This construction is quite standard. For example, it was used in~\cite{KGWZ02,CH10,ADMW11}. 

In what follows, we suppose that $\Psi$ is the identity map and
\begin{align}\label{man_M_def}
M=[0,\sigma]\times \mathcal G/\mathcal H
\end{align}
for some $\sigma>0$.  
This does not cause any loss of generality. Let $R$ be a symmetric positive-definite
$\mathcal G$-invariant $(0,2)$-tensor field on $\partial M$. We will use $R$ to impose a boundary condition on~\eqref{intro_PRC}. Denote by $R^0$ and $R^\sigma$ the restrictions of $R$ to $\{0\}\times \mathcal G/\mathcal H$ and
$\{\sigma\}\times
\mathcal G/\mathcal H$, respectively.
Given $\tau\in[0,\sigma]$, it will be convenient for us to identify the tangent spaces to $\{\tau\}\times\mathcal G/\mathcal H$  at the point $\{\tau\}\times\mathcal H$ with $\mathfrak p$ in the natural way. We observe that $R$ is fully determined by how
$R^0$ and $R^\sigma$ act on~$\mathfrak
p$. Thanks to Hypothesis~\ref{assum_decomp_p}, there exist positive numbers $a_1,\ldots,a_n$ and
$b_1,\ldots,b_n$ satisfying the equalities
\begin{align}\label{R_form}
R^0(X,Y)&=a_1^2\,Q\big(\pr_{\mathfrak
p_1}X,\pr_{\mathfrak p_1}Y\big)+\cdots+a_n^2\,Q\big(\pr_{\mathfrak
p_n}X,\pr_{\mathfrak p_n}Y\big), \notag \\
R^\sigma(X,Y)&=b_1^2\,Q\big(\pr_{\mathfrak
p_1}X,\pr_{\mathfrak p_1}Y\big)+\cdots+b_n^2\,Q\big(\pr_{\mathfrak
p_n}X,\pr_{\mathfrak p_n}Y\big),\qquad X,Y\in\mathfrak p.
\end{align}

Fix a number $\alpha>0$ such that
\begin{align}\label{bds_phi}
\big|\phi_i(r)\big|&\le\alpha,\qquad
i=1,\ldots,n,~r\in[0,\sigma],
\end{align}
along with a pair of numbers $\omega_1,\omega_2>0$ such that
\begin{align}\label{bds_ab}
\omega_1\le a_i,b_i\le\omega_2,\qquad i=1,\ldots,n.
\end{align}
Denote by $d_i$ the dimension of $\mathfrak p_i$ for $i=1,\ldots,n$. Given a Riemannian metric $G$ defined on a neighbourhood of $\partial M$, we write $G_{\partial M}$ for the metric
induced by $G$ on $\partial M$. Our first result is a sufficient condition for the global solvability of a Dirichlet-type problem for~\eqref{intro_PRC}.

\begin{theorem}\label{thm_main}
There exist functions $\rho_0:(0,\infty)^2\to(0,\infty)$ and
$\sigma_0:(0,\infty)^5\to(0,\infty)$, both independent of the tensor
fields $T$ and $R$, such that the following statement is satisfied:
if the formulas
\begin{align}\label{hyp_thm_deriv}
\bigg|\frac
d{dr}\phi_i(r)\bigg|&\le c_1\sigma,\qquad |a_i-b_i|\le c_2\sigma^2,\qquad i=1,\ldots,n,~r\in[0,\sigma],
\end{align}
and the formulas
\begin{align}\label{hyp_thm_bounds}
\sum_{i=1}^nd_i\bigg(\max\{\phi_i(r),0\}
+\frac{\omega_2^2}{\omega_1^2}\min\{\phi_i(r),0\}\bigg)&>\rho_0(\omega_1,\omega_2),\notag \\
\sigma&<\sigma_0(\alpha,\omega_1,\omega_2,c_1,c_2),\qquad r\in[0,\sigma],
\end{align}
hold for some $c_1,c_2>0$, then the manifold $M$ supports a $\mathcal G$-invariant Riemannian
metric $G$ solving the equation $\Ric(G)=T$ on $M$
under the boundary condition $G_{\partial M}=R$.
\end{theorem}

\begin{remark}
When proving Theorem~\ref{thm_main}, we will obtain explicit
expressions for $\rho_0$ and $\sigma_0$. These expressions (at least
the one for $\sigma_0$) will be rather unsightly.
\end{remark}

\begin{remark}\label{rem_intui}
The first formula in~\eqref{hyp_thm_deriv} essentially forbids the
part of $T$ tangent to the $\mathcal G$-orbits to change dramatically from
one orbit to another. The second one says that $R^0$
should not be very different from~$R^\sigma$. Note that formulas~\eqref{hyp_thm_deriv} are
automatically satisfied when $R^0$ coincides with
$R^\sigma$ on $\mathfrak p$ and $\phi_1,\ldots,\phi_n$ are
constant. The meaning of~\eqref{hyp_thm_bounds} is that the
tensor field $T$ has to be large in the directions tangent to the
$\mathcal G$-orbits and small in the direction transverse to the
$\mathcal G$-orbits.
\end{remark}

Our second result establishes the local solvability of~\eqref{intro_PRC} in the interior of $M$. Given $\tau\in[0,\sigma]$,
we denote by $\Gamma^\tau$ the $\mathcal G$-orbit $\{\tau\}\times \mathcal G/\mathcal H$
on $M$.

\begin{theorem}\label{thm_loc_ex_int}
For each $\tau\in(0,\sigma)$, there exists a $\mathcal G$-invariant
Riemannian metric $G^\tau$ on $M$ such that the equaity $\Ric(G^\tau)=T$ holds on some neighbourhood of $\Gamma^\tau$.
\end{theorem}

Next, we establish the local solvability of~\eqref{intro_PRC} near $\partial M$.

\begin{theorem}\label{thm_loc_ex_bdy}
There exists a $\mathcal G$-invariant Riemannian metric $G^{bdy}$ on $M$ such that
$\Ric(G^{bdy})=T$ on some neighbourhood of $\partial M$ and $G_{\partial M}^{bdy}=R$.
\end{theorem}

The proofs of Theorems~\ref{thm_loc_ex_int} and~\ref{thm_loc_ex_bdy}
will rely on Proposition~\ref{prop_local_ex} appearing below.  This proposition will also demonstrate that a $\mathcal G$-invariant metric solving the prescribed Ricci
curvature equation near a $\mathcal G$-orbit on $M$ is uniquely determined by the metric it induces on this $\mathcal G$-orbit and by the orbit's second fundamental form.

\begin{remark}
Assume $\mathcal G$ is the special orthogonal group $\mathrm{SO}(d)$ and
$M_1$ coincides with $\mathbb R^d$ less a closed ball around the origin. One
may then be able to study
boundary-value problems for~\eqref{intro_PRC} with the methods
of~\cite{JCDDT94}; see also~\cite{RH84}. These methods consist in reducing the prescribed Ricci curvature equation to a first-order ordinary differential equation for a single real-value function. The authors of~\cite{JCDDT94} were able to achieve such a reduction by exploiting the fact that $\mathrm{SO}(d)$-invariant metrics on $\mathbb R^d$ are globally conformally flat. In essence, their arguments relied on a clever change of variable in the prescribed Ricci curvature equation.
\end{remark}

\begin{remark}\label{rem_abelian}
Instead of requiring that Hypothesis~\ref{assum_decomp_p} hold for $\mathcal G/\mathcal H$, we may assume $\mathcal G/\mathcal H$ is an abelian Lie group.
The $\mathcal H$-module $\mathfrak p$ can then be written in the
form~\eqref{p_decomp} with the $\mathcal H$-modules $\mathfrak p_k$ being
one-dimensional for all $k=1,\ldots,n$. As before, we suppose there exists a
diffeomorphism $\Psi$ satisfying formula~\eqref{T_form1}. In our
current situation, however, it is not necessarily the case that
there are smooth functions $\phi_1,\ldots,\phi_n$ from
$[0,\sigma]$ to $\mathbb R$ obeying equality~\eqref{T_form2}. Assume
that such functions do exist. Suppose also that one can find
positive numbers $a_1,\ldots,a_n$ and $b_1,\ldots,b_n$ such
that~\eqref{R_form} holds. Thus, we demand that $T$ and $R$ be
diagonal with respect to~\eqref{p_decomp}. It is then possible to
prove the assertions of Theorems~\ref{thm_loc_ex_int},
\ref{thm_loc_ex_bdy}, and~\ref{thm_main} using the reasoning of
Section~\ref{sec_proofs}.
\end{remark}

\begin{remark}
Instead of assuming the existence of $\Psi$ above, one may
assume there is a diffeomorphism such
that~\eqref{T_form1} holds with this diffeomorphism substituted for
$\Psi$ and a minus sign in front of $dr\otimes dr$.
The techniques in the present paper seem to be effective for
treating this case. We will not dwell on any further details.
\end{remark}

\begin{example}
Given an interval $I\subset\mathbb R$ and a number $\epsilon\ge0$, denote
\begin{align*}
\mathcal D_I=\big\{(x,y)\in\mathbb
R^2\,|\,\sqrt{x^2+y^2}\in I\big\},\qquad \mathcal S_\epsilon=\{(x,y)\in\mathbb
R^2\,|\,x^2+y^2=\epsilon^2\}.
\end{align*}
Assume $\mathcal G$ is the product
$\mathrm{SO}(2)\times\mathrm{SO}(2)$. Define $M_1$ to equal $\mathcal D_{(\frac\chi2,2)}\times\mathcal S_1$ with $\chi>0$. The standard action of $\mathrm{SO}(2)$ on $\mathbb R^2$ gives rise to
an action of $\mathcal G$ on $M_1$. The orbits of this action are the tori
$\mathcal S_\epsilon\times\mathcal S_1$ with $\epsilon\in\big(\frac\chi2,2\big)$. The isotropy group of an arbitrarily chosen point in $M_1$ consists of nothing but the identity element in $\mathcal G$. Consider a $(0,2)$-tensor field $T$ on~$M_1$. It is
convenient for us to assume that $T$ is positive-definite, although
this assumption can be relaxed. Suppose~$T$ is rotationally symmetric in the sense
of~\cite{JDDKAY10,AP11}. This means $T$ is $\mathcal G$-invariant and
diagonal with respect to the cylindrical coordinates on $\mathcal
D_{[0,2]}\times\mathcal S_1$. We define $M$ to equal $\mathcal D_{[\chi,1]}\times\mathcal S_1$. Thus, $M$ is a solid torus less a neighbourhood of the core circle. Consider a symmetric positive-definite $(0,2)$-tensor field $R$ on $\partial M$. We suppose $R$ is $\mathcal G$-invariant and
diagonal in the coordinates induced on $\partial M$ by the
cylindrical coordinates on $\mathcal
D_{[0,2]}\times\mathcal S_1$. In the current setting, Theorem~\ref{thm_main} (along with Remark~\ref{rem_abelian}) yields a sufficient condition for the solvability of the equation $\Ric(G)=T$ on all of $M$, subject to $G_{\partial M}=R$. No such condition previously appeared in the literature. Theorems~\ref{thm_loc_ex_int}
and~\ref{thm_loc_ex_bdy} imply local solvability; cf.~\cite{AP11}.
\end{example}

\section{The proofs}\label{sec_proofs}

In what follows, we assume $T$ is positive-definite and $c_1=c_2=1$. Thus, the function $\sigma_0$, whose existence Theorem~\ref{thm_main} asserts, becomes a function of three variables, not five. These assumptions will make our arguments easier to follow. Removing them is straightforward. 

\subsection{Preparatory material}

We begin by stating a formula for the Ricci curvature of a $\mathcal G$-invariant metric on $M$. This formula will involve two arrays of numbers, $(\beta_k)_{k=1}^n$ and $\big(\gamma_{k,l}^m\big)_{k,l,m=1}^n$. In order to introduce them, denote by $[\cdot,\cdot]$ and $P$ the Lie bracket and the
Killing form of the Lie algebra $\mathfrak g$. The irreducibility of the summands in decomposition~\eqref{p_decomp} implies the existence of nonnegative numbers
$\beta_1,\ldots,\beta_n$
such that
\begin{align*}
P(X,Y)=-\beta_kQ(X,Y),\qquad k=1,\ldots,n,~X,Y\in\mathfrak p_k.
\end{align*}
Because the group $\mathcal G$ is compact and Hypothesis~\ref{assum_decomp_p} holds, at least one of these numbers must be strictly positive. Suppose $d$ is the
dimension of $M$. We choose a $Q$-orthonormal basis $(\tilde
e_i)_{i=1}^{d-1}$ of the space $\mathfrak p$ adapted to~\eqref{p_decomp}. In addition to $\beta_1,\ldots,\beta_n$, let us define
\begin{align*}
\gamma_{k,l}^m=\frac1{d_k}\sum Q([\tilde e_{\iota_k},\tilde
e_{\iota_l}],\tilde
e_{\iota_m})^2
\end{align*}
for $m,k,l=1,\ldots,n$. The sum here is taken over all $\iota_k$, $\iota_l$, and $\iota_m$ such that $\tilde e_{\iota_k}\in\mathfrak p_k$, $\tilde e_{\iota_l}\in\mathfrak p_l$, and $\tilde e_{\iota_m}\in\mathfrak p_m$.
Note that $\gamma_{k,l}^m$ is independent of the choice of $(\tilde e_i)_{i=1}^{d-1}$.

Consider a Riemannian metric $G$ on $M$. Suppose
$h,f_1,\ldots,f_n$ are smooth functions from $[0,\sigma]$ to
$(0,\infty)$. Let $G$ be defined by the equality
\begin{align}\label{cal_G_form1}
G=h^2(r)\,dr\otimes dr+G_r,\qquad
r\in[0,\sigma].
\end{align}
The tensor field $G_r$ in the right-hand side is the
$\mathcal G$-invariant Riemannian metric on $\mathcal G/\mathcal H$ such that
\begin{align}\label{cal_G_form2}
G_r(X,Y)=f_1^2(r)\,Q\big(\pr_{\mathfrak
p_1}X,\pr_{\mathfrak
p_1}Y\big)+\cdots+f_n^2(r)\,Q\big(\pr_{\mathfrak
p_n}X,\pr_{\mathfrak p_n}Y\big),\qquad X,Y\in\mathfrak p.
\end{align}
In the sequel, the prime
next to a real-valued function on $[0,\sigma]$ will denote the
derivative of this function.

\begin{lemma}\label{lemma_Ricci_comp}
The Ricci curvature of the Riemannian metric $G$
given by~\eqref{cal_G_form1} and~\eqref{cal_G_form2} obeys the
equality
\begin{align*}
\Ric(G)=&\Ric^{\lin}+\Ric_r^{\orb},\qquad r\in[0,\sigma],
\end{align*}
where $\Ric^{\lin}$ is the (0,2)-tensor field on $[0,\sigma]$ satisfying
\begin{align*}
\Ric^{\lin}=-\sum_{k=1}^nd_k\left(\frac{f_k''}{f_k}-\frac{h'f_k'}{hf_k}\right)dr\otimes dr
\end{align*}
and $\Ric_r^{\orb}$ is the $\mathcal G$-invariant (0,2)-tensor field on
$\mathcal G/\mathcal H$ satisfying
\begin{align*}
\Ric_r^{\orb}&(X,Y) \\ &=\sum_{i=1}^n\Bigg(\frac{\beta_i}2+\sum_{k,l=1}^n\gamma_{i,k}^l\frac{f_i^4-2f_k^4}{4f_k^2f_l^2}
-\frac{f_if_i'}h\sum_{k=1}^n
d_k\frac{f_k'}{hf_k}+\frac{f_i'^2}{h^2}-\frac{f_if_i''}{h^2}+\frac{f_ih'f_i'}{h^3}\Bigg)Q\big(\pr_{\mathfrak
p_i}X,\pr_{\mathfrak p_i}Y\big)
\end{align*}
for $X,Y\in\mathfrak p$.
\end{lemma}

\begin{proof}
The terms involving $dr$ are computed and listed under Proposition~1.14 in~\cite{KGWZ02}. Let us find $\Ric_r^{\orb}$. Hypothesis~\ref{assum_decomp_p} implies
\begin{align*}
\Ric_r^{\orb}(X,Y)=0
\end{align*}
when $X\in\mathfrak p_i$ and $Y\in\mathfrak p_j$ for $i,j=1,\ldots,n$ such that $i\ne j$. Remark~1.16 in~\cite{KGWZ02} states that
\begin{align*}
\Ric_r^{\orb}(X,X)=\Bigg(\frac{\beta_i}2+\sum_{k,l=1}^n\gamma_{i,k}^l\frac{f_i^4-2f_k^4}{4f_k^2f_l^2}
-\frac{f_if_i'}h\sum_{k=1}^n
d_k\frac{f_k'}{hf_k}+\frac{f_i'^2}{h^2}-\frac{f_if_i''}{h^2}+\frac{f_ih'f_i'}{h^3}\Bigg)Q(X,X)
\end{align*}
when $X\in\mathfrak p_i$ and $i=1,\ldots,n$.
In view of Hypothesis~\ref{assum_decomp_p}, the desired expression for $\Ric_r^{\orb}$ immediately follows.
\end{proof}

If the Ricci
curvature of $G$ coincides with $T$, then
Lemma~\ref{lemma_Ricci_comp} yields the equalities
\begin{align}\label{Ric_ODE}
-\sum_{k=1}^nd_k\left(\frac{f_k''}{f_k}-\frac{h'f_k'}{hf_k}\right)&=1,\notag
\\
\frac{\beta_i}2+\sum_{k,l=1}^n\gamma_{i,k}^l\frac{f_i^4-2f_k^4}{4f_k^2f_l^2}
-\frac{f_if_i'}h\sum_{k=1}^nd_k\frac{f_k'}{hf_k}+\frac{f_i'^2}{h^2}-\frac{f_if_i''}{h^2}+\frac{f_ih'f_i'}{h^3}&=\phi_i,\qquad
i=1,\ldots,n.
\end{align}
The following result is essentially a restatement of the contracted second Bianchi identity.

\begin{lemma}\label{lemma_Bianchi}
Assume the Ricci curvature of the metric $G$ given
by~\eqref{cal_G_form1} and~\eqref{cal_G_form2} obeys the equality
\begin{align*}
\Ric(G)=\bar\sigma(r)\,dr\otimes dr+T_r,\qquad
r\in[0,\sigma],
\end{align*}
with $\bar\sigma$ being a smooth function on $[0,\sigma]$ and $T_r$ satisfying~\eqref{T_form2}. Then
\begin{align}\label{Bianchi_lem_ineq}
\frac{\bar\sigma'}{2h^2}-\frac{\bar\sigma
h'}{h^3}=\sum_{k=1}^nd_k\left(\frac{\phi_k'}{2f_k^2}-\frac{\bar\sigma
f_k'}{h^2f_k}\right).
\end{align}
\end{lemma}

\begin{proof}
Fix a $Q$-orthonormal basis $(\tilde e_i)_{i=1}^{d-1}$ of the space
$\mathfrak p$ adapted to the decomposition~\eqref{p_decomp}. Recall
that we identify $\mathfrak p$ with the tangent space of $\mathcal G/\mathcal H$ at
$\mathcal H$. Given $r_0\in[0,\sigma]$, let us construct a $\mathcal G$-invariant
$G$-orthonormal frame field $(e_i)_{i=1}^d$ on a
neighbourhood $U$ of $(r_0,\mathcal H)$ in $M$ so that the following requirements are met:
\begin{enumerate}
\item
The equality $e_i=\big(0,\frac1{f_i(r)}\tilde e_i\big)$ holds at
$(r,\mathcal H)$ for every $i=1,\ldots,d-1$ as long as
$(r,\mathcal H)\in U$.
\item
The vector field $e_d$ coincides with
$\big(\frac1{h(r)}\frac\partial{\partial r},0\big)$ on $U$.
\end{enumerate}
The contracted second Bianchi identity then implies
\begin{align*}
\sum_{i=1}^d\nabla_{e_i}\Ric(G)(e_i,e_d)=\frac12e_d\Bigg(\sum_{i=1}^d\Ric(G)(e_i,e_i)\Bigg).
\end{align*}
The symbol $\nabla$ in the left-hand side denotes the covariant
derivative in the tensor bundle over $M$ given by the
Levi-Civita connection of $G$. We calculate and see
that the equalities
\begin{align*}
\sum_{i=1}^d\nabla_{e_i}\Ric(G)(e_i,e_d)&=\sum_{i=1}^d e_i(\Ric(G)(e_i,e_d))
-\sum_{i=1}^d\Ric(G)(\nabla_{e_i}e_i,e_d)
\\ &\hphantom{=}~
-\sum_{i=1}^d\Ric(G)(e_i,\nabla_{e_i}e_d) \\
&=e_d(\Ric(G)(e_d,e_d)) -\sum_{i=1}^dG(\nabla_{e_i}e_i,e_d)\Ric(G)(e_d,e_d)
\\
&\hphantom{=}~-\sum_{i=1}^dG(\nabla_{e_i}e_d,e_i)\Ric(G)(e_i,e_i) \\
&=\frac{\bar\sigma'}{h^3}-\frac{2\bar\sigma
h'}{h^4}+\sum_{k=1}^nd_k\frac{\bar\sigma
f_k'}{h^3f_k}-\sum_{k=1}^nd_k\frac{f_k'}{hf_k^3}\phi_k,
\end{align*}
as well as the equality
\begin{align*}
\frac12e_d\left(\sum_{i=1}^d\Ric(G)(e_i,e_i)\right)=\sum_{k=1}^nd_k\left(\frac{\phi_k'}{2hf_k^2}-\frac{f_k'}{hf_k^3}\phi_k\right)
+\frac{\bar\sigma'}{2h^3}-\frac{\bar\sigma h'}{h^4}\,,
\end{align*}
hold at $(r_0,\mathcal H)$. The assertion of the lemma
follows immediately.
\end{proof}

Denote by $f$ and $\phi$ the
functions $(f_1,\ldots,f_n)$ and $(\phi_1,\ldots,\phi_n)$ acting from
$[0,\sigma]$ to $(0,\infty)^n$.
We can rewrite the second equality in~\eqref{Ric_ODE} as
\begin{align}\label{sys_to_solve}
f''(r)&=F(h(r),h'(r),f(r),f'(r),\phi(r)),\qquad r\in[0,\sigma],
\end{align}
with $F:(0,\infty)\times\mathbb R\times(0,\infty)^n\times\mathbb
R^{n+n}\to\mathbb R^n$ given by the formulas
\begin{align*}
F(p,q,x,y,z)&=(F_1(p,q,x,y,z),\ldots,F_n(p,q,x,y,z)), \\
F_i(p,q,x,y,z)&=\frac{\beta_ip^2}{2x_i}+p^2\sum_{k,l=1}^n\gamma_{i,k}^l\frac{x_i^4-2x_k^4}{4x_ix_k^2x_l^2}
-\sum_{k=1}^nd_k\frac{y_iy_k}{x_k}+\frac{y_i^2}{x_i}+\frac{qy_i}p-\frac{p^2}{x_i}z_i,\qquad
i=1,\ldots,n, \\
p\in(0,\infty),~q&\in\mathbb
R,~x=(x_1,\ldots,x_n)\in(0,\infty)^n,~y=(y_1,\ldots,y_n)\in\mathbb
R^n,~z=(z_1,\ldots,z_n)\in\mathbb R^n.
\end{align*}
The prime next to a vector-valued function means component-wise differentiation. Combining the two equalities in~\eqref{Ric_ODE}, we find
\begin{align}\label{sys_to_slv(h)}
H_1(f(r),f'(r))=h^2(r)H_2(f(r),\phi(r)),\qquad r\in[0,\sigma],
\end{align}
with the mappings $H_1:(0,\infty)^n\times\mathbb R^n\to\mathbb R$
and $H_2:(0,\infty)^n\times\mathbb R^n\to\mathbb R$ defined by the
formulas
\begin{align*}
H_1(x,y)&=1-\sum_{k=1}^nd_k\left(\sum_{l=1}^nd_l\frac{y_ky_l}{x_kx_l}
-\frac{y_k^2}{x_k^2}\right),~
H_2(x,z)=\sum_{k=1}^nd_k\Bigg(\frac{z_k}{x_k^2}-\frac{\beta_k}{2x_k^2}
-\sum_{l,m=1}^n\gamma_{k,l}^m\frac{x_k^4-2x_l^4}{4x_k^2x_l^2x_m^2}\Bigg),
\\ x&=(x_1,\ldots,x_n)\in(0,\infty)^n,~y=(y_1,\ldots,y_n)\in\mathbb
R^n,~z=(z_1,\ldots,z_n)\in\mathbb R^n.
\end{align*}
It will be convenient for us to denote 
\begin{align*}
H(x,y,z)=\sqrt{H_1(x,y)H_2^{-1}(x,z)}
\end{align*}
for $x\in(0,\infty)^n$, $y\in\mathbb R^n$, and $z\in\mathbb R^n$
such that $H_2(x,z)\ne0$ and $H_1(x,y)H_2^{-1}(x,z)\ge0$.

Solving~\eqref{Bianchi_lem_ineq} for $h'(r)$ and substituting~1 for $\bar\sigma(r)$, we arrive at the following conclusion: if $\Ric(G)$ coincides with $T$, then
\begin{align}\label{eq_Bianchi_K}
h'(r)=K(h(r),f(r),f'(r),\phi'(r)),\qquad r\in[0,\sigma].
\end{align}
Here, $K:(0,\infty)^{1+n}\times\mathbb R^{n+n}\to\mathbb R$ is given
by
\begin{align*}
K(p,x,y,w)&=\sum_{i=1}^nd_i\left(\frac{py_i}{x_i}-\frac{p^3w_i}{2x_i^2}\right),\\
p&\in(0,\infty),~x=(x_1,\ldots,x_n)\in(0,\infty)^n,~y=(y_1,\ldots,y_n)\in\mathbb
R^n,~w=(w_1,\ldots,w_n)\in\mathbb R^n.
\end{align*}

Let $a$ and $b$ denote the vectors $(a_1,\ldots,a_n)$ and
$(b_1,\ldots,b_n)$ with the numbers $a_1,\ldots,a_n$ and
$b_1,\ldots,b_n$ coming from~\eqref{R_form}. If the metric $G_{\partial M}$ induced by $G$ on $\partial M$ equals $R$,
then
\begin{align}\label{BC_on_f}
f(0)=a,\qquad f(\sigma)=b.
\end{align}
We also point out that, whenever~\eqref{sys_to_slv(h)} holds, we
must have
\begin{align}\label{BC_on_h}
H_1(f(0),f'(0))=h^2(0)H_2(f(0),\phi(0)).
\end{align}

\subsection{Proof of Theorem~\ref{thm_main} (less the key lemma)}\label{subsec_less_lemma}

Intuitively, our plan for proving Theorem~\ref{thm_main} is to find a metric $G$ satisfying two
requirements. The first one is that $\Ric(G)$ equal $T$ in
the directions tangent to the $\mathcal G$-orbits. The other is that
$G$ and $T$ obey the contracted second Bianchi identity.
When both of these requirements are met, it must be the case that
$\Ric(G)=T$. We define $\rho_0$ by the formula
\begin{align*}
\rho_0(p,q)=2\sum_{k=1}^nd_k\Bigg(\frac{\beta_kq^2}{2p^2}
+\sum_{l,m=1}^n\gamma_{k,l}^m\frac{q^6}{4p^6}\Bigg),\qquad p,q\in(0,\infty).
\end{align*}

\begin{lemma}\label{lemma_exist_hf}
Assume that inequalities~\eqref{hyp_thm_deriv} and the first inequality in~\eqref{hyp_thm_bounds} are satisfied. There exists a
function $\sigma_0:(0,\infty)^3\to(0,\infty)$ such that the
following statement holds: if $\sigma$ is less than
$\sigma_0(\alpha,\omega_1,\omega_2)$, then we can find smooth
$f:[0,\sigma]\to(0,\infty)^n$ and $h:[0,\sigma]\to(0,\infty)$
solving equations~\eqref{sys_to_solve} and \eqref{eq_Bianchi_K} under
the boundary conditions~\eqref{BC_on_f} and~\eqref{BC_on_h}.
\end{lemma}

We will present the proof of Lemma~\ref{lemma_exist_hf} in
Section~\ref{subsec_pf_lemma}. Meanwhile, fix a function $\sigma_0$
satisfying the assertion of this lemma. Suppose $\sigma$ is less
than $\sigma_0(\alpha,\omega_1,\omega_2)$. Let
$f:[0,\sigma]\to(0,\infty)^n$ and $h:[0,\sigma]\to(0,\infty)$ be
smooth functions obeying~\eqref{sys_to_solve}, \eqref{eq_Bianchi_K}, \eqref{BC_on_f} and~\eqref{BC_on_h}. We define the metric $G$ on $M$ through~\eqref{cal_G_form1}--\eqref{cal_G_form2}. It is
easy to see that the Ricci curvature of $G$ must
equal
\begin{align*}
\bar\sigma(r)\,dr\otimes dr+T_r,\qquad r\in[0,\sigma],
\end{align*}
for some $\bar\sigma:[0,\sigma]\to\mathbb R$. The induced metric
$G_{\partial M}$ coincides with $R$. The proof of Theorem~\ref{thm_main} will be
complete if we demonstrate that $\bar\sigma(r)=1$ for all $r\in[0,\sigma]$.

Lemma~\ref{lemma_Bianchi} implies
\begin{align}\label{eq_w_Bianchi}
\bar\sigma'=\frac{2\bar\sigma
h'}h+\sum_{i=1}^nd_i\left(\frac{h^2\phi_i'}{f_i^2}-\frac{2\bar\sigma
f_i'}{f_i}\right).
\end{align}
Thanks to~\eqref{eq_Bianchi_K}, this formula will still hold if we replace $\bar\sigma$ in it by the function identically equal to~1 on $[0,\sigma]$. Furthermore, invoking Lemma~\ref{lemma_Ricci_comp} and the boundary
conditions~\eqref{BC_on_f}--\eqref{BC_on_h}, we find
\begin{align*}
\bar\sigma(0)&=-\sum_{k=1}^nd_k\left(\frac{f_k''(0)}{a_k}-\frac{h'(0)f_k'(0)}{h(0)a_k}\right) \\
&=-\sum_{k=1}^nd_k\left(\frac{\beta_kh^2(0)}{2a_k^2}+h^2(0)\sum_{l,m=1}^n\gamma_{k,l}^m\frac{a_k^4-2a_l^4}{4a_k^2a_l^2a_m^2}
-\sum_{l=1}^nd_l\frac{f_k'(0)f_l'(0)}{a_ka_l}+\frac{f_k'^2(0)}{a_k^2}-\frac{h^2(0)}{a_k^2}\phi_k(0)\right)
\\ &=h^2(0)H_2(a,\phi(0))
+(1-H_1(a,f'(0)))=1.
\end{align*}
The standard theorems on the uniqueness of solutions to ordinary differential equations now yield $\bar\sigma(r)=1$ for $r\in[0,\sigma]$.

\subsection{Proof of Lemma~\ref{lemma_exist_hf}}\label{subsec_pf_lemma}

From now on and until the end of
Section~\ref{subsec_pf_lemma}, we assume that inequalities~\eqref{hyp_thm_deriv} and the first inequality
in~\eqref{hyp_thm_bounds} are
satisfied. Let
$\bar f:[0,\sigma]\to\mathbb R^n$ be defined by
\begin{align*}
\bar f(r)=a\frac{\sigma-r}\sigma+b\frac r\sigma\,,\qquad
r\in[0,\sigma].
\end{align*}
We seek the function $f$, whose existence Lemma~\ref{lemma_exist_hf} asserts, in a neighbourhood of $\bar f$. This, in particular, will help us ensure the positivity of the components of $f$. Similarly, we look for the function $h$ in a neighbourhood of the function $\bar h$ to be introduced in Lemma~\ref{lemma_bar_h}. Our arguments will involve the constants
\begin{align*}
H_0&=H\big(\bar f(0),\bar
f'(0),\phi(0)\big), \\
K_0&=\sup_{p\in\left[\frac{H_0}2,\frac{3H_0}2\right]}\sup_{r\in[0,\sigma]}\big|K\big(p,\bar
f(r),\bar f'(r),\phi'(r)\big)\big|.
\end{align*}
The second inequality in~\eqref{hyp_thm_deriv} and the
first inequality in~\eqref{hyp_thm_bounds} imply that $H_0$ is well-defined under the assumptions of Lemma~\ref{lemma_bar_h}. Recall that the letter $d$ stands for the dimension of $M$. It
is evident that $\sum_{i=1}^nd_i=d-1$.

\begin{lemma}\label{lemma_bar_h}
Let $\rho_1,\sigma_1>0$ be given by the formulas
\begin{align*}
\rho_1&=\max\left\{4\Bigg(\sum_{k=1}^nd_k\Bigg(\frac\alpha{\omega_1^2}
+\sum_{l,m=1}^n\gamma_{k,l}^m\frac{\omega_2^4}{2\omega_1^6}\Bigg)\Bigg)^{\frac12}
,\frac94\left(\frac{\rho_0(\omega_1,\omega_2)}{2\omega_2^2}\right)^{-\frac12}\right\},
\\ \sigma_1&=\min\left\{1,\frac{\omega_1}{4d}\,,
\frac{2\omega_1^2}{\big(2\rho_1^2\omega_1+\rho_1^4\big)(d-1)}\right\}.
\end{align*}
If $\sigma\le\sigma_1$, then the problem
\begin{align}\label{eq_bar_h}
\bar h'(r)&=K\big(\bar
h(r),\bar f(r),\bar f'(r),\phi'(r)\big),\qquad r\in[0,\sigma], \notag \\
\bar h(0)&=H_0,
\end{align}
has a unique smooth solution $\bar
h:[0,\sigma]\to\big(\frac1{\rho_1},\rho_1\big)$.
\end{lemma}

\begin{proof}
Assume $\sigma<\sigma_1$. Employing the standard theory of ordinary differential equations, it is easy to show that the
problem~\eqref{eq_bar_h} has a unique smooth solution on the
interval
$\left[0,\min\left\{\sigma,\frac{H_0}{2\varkappa}\right\}\right]$ as long as $K_0\le\varkappa$. The
values of this solution must lie in
$\left[\frac{H_0}2,\frac{3H_0}2\right]$.

Our assumptions imply
\begin{align*}
\frac1{\rho_1}\le\frac{H_0}2<\frac{3H_0}2\le\rho_1.
\end{align*}
In view of~\eqref{hyp_thm_deriv}, the estimate
\begin{align*}
K_0\le\frac{\big(2\rho_1\omega_1\sigma+\rho_1^3\sigma\big)(d-1)}{2\omega_1^2}
\end{align*}
holds true. Keeping these facts in mind, we conclude that
problem~\eqref{eq_bar_h} has a unique smooth solution
\begin{align*}
\bar
h:\left[0,\min\left\{\sigma,\frac{2\omega_1^2}{\big(2\rho_1^2\omega_1\sigma+\rho_1^4\sigma\big)(d-1)}\right\}\right]
\to\left(\frac1{\rho_1},\rho_1\right).
\end{align*}
At the same time, whenever $\sigma\le\sigma_1$, the equality
\begin{align*}
\sigma=\min\left\{\sigma,
\frac{2\omega_1^2}{\big(2\rho_1^2\omega_1\sigma+\rho_1^4\sigma\big)(d-1)}\right\}
\end{align*}
is satisfied. This means $\bar h$ is actually defined on
$[0,\sigma]$.
\end{proof}

From this moment on and until
the end of Section~\ref{subsec_pf_lemma}, let us assume that
$\sigma\le\sigma_1$. It then makes sense to talk about $\bar h$. Our
plan is to prove, for small $\sigma$, the existence of smooth
$u:[0,\sigma]\to\mathbb R^n$ and $v:[0,\sigma]\to\mathbb R$ solving
the equations
\begin{align}\label{sys_centered}
u''(r)&=F\big(\bar h(r)+v(r),\bar h'(r)+v'(r),\bar f(r)+u(r),\bar
f'(r)+u'(r),\phi(r)\big),\notag \\
v'(r)&=-\bar h'(r)+K\big(\bar h(r)+v(r),\bar f(r)+u(r),\bar
f'(r)+u'(r),\phi'(r)\big),\qquad r\in[0,\sigma],
\end{align}
under the boundary conditions
\begin{align}\label{BC_centered}
u(0)&=u(\sigma)=0, \notag \\
v(0)&=-\bar h(0)+H\big(\bar f(0)+u(0),\bar f'(0)+u'(0),\phi(0)\big).
\end{align}
We will then set $f=\bar f+u$ and $h=\bar h+v$. It is obvious that
these functions will obey~\eqref{sys_to_solve}, \eqref{eq_Bianchi_K}, \eqref{BC_on_f} and~\eqref{BC_on_h}.

Our proof of the existence of $u$ and $v$ will rely on the Schauder
fixed point theorem. Let us introduce the space $\mathcal B$ of all
the pairs $(\upsilon_1,\upsilon_2)$ such that
$\upsilon_1:[0,\sigma]\to\mathbb R^n$ is $C^1$-differentiable and
$\upsilon_2:[0,\sigma]\to\mathbb R$ is continuous. We endow
$\mathcal B$ with the norm
\begin{align*}
|(\upsilon_1,\upsilon_2)|_{\mathcal
B}=\sup_{r\in[0,\sigma]}|\upsilon_1(r)|_{\mathbb
R^n}+\sigma\sup_{r\in[0,\sigma]}|\upsilon_1'(r)|_{\mathbb
R^n}+\sup_{r\in[0,\sigma]}|\upsilon_2(r)|,
\end{align*}
where $|\cdot|_{\mathbb R^n}$ is the Euclidean norm in $\mathbb
R^n$. Denote by $B(L)$ the closed ball in $\mathcal B$ of radius $L>0$
centered at~0. We will now define a map $\mathcal C:B(L)\to\mathcal B$
and show that $\mathcal C$ has a fixed point $(u,v)$ under
appropriate conditions. The functions $u$ and $v$ will
satisfy~\eqref{sys_centered} and~\eqref{BC_centered}.

Assume the radius $L$ is less than or equal to
$\frac\sigma2\min\big\{\omega_1,\frac1{\rho_1}\big\}$. Given
$(\mu,\nu)\in B(L)$, let $u_{\mu,\nu}$ be the unique solution of the
problem
\begin{align}\label{eq_centered_xi}
u_{\mu,\nu}''(r)&=\bar F\big(\bar h(r)+\nu(r),\bar
f(r)+\mu(r),\bar
f'(r)+\mu'(r),\phi(r),\phi'(r)\big),\qquad r\in[0,\sigma], \notag \\
u_{\mu,\nu}(0)&=u_{\mu,\nu}(\sigma)=0,
\end{align}
where
\begin{align*}
\bar F(p,x,y,z,w)&=F(p,K(p,x,y,w),x,y,z),\qquad
p\in(0,\infty),~x\in(0,\infty)^n,~y,z,w\in\mathbb R^n.
\end{align*}
It is obvious that such a solution exists. Moreover, it is easy to write down an explicit formula for it (the formula is quite lengthy, and we will not present it here; the reader may find it in, e.g.,~\cite[Section~XII.4]{PH64}). We will set $\mathcal
C(\mu,\nu)=(u_{\mu,\nu},v_{\mu,\nu})$ for a properly chosen
$v_{\mu,\nu}:[0,\sigma]\to\mathbb R$. Before we can
describe~$v_{\mu,\nu}$, however, we need to state the following
auxiliary result.

\begin{lemma}\label{lemma_bd_F_Theta}
Let $\Theta$ be given by the formula
\begin{align*}
\Theta=\sqrt{d}\max_{i=1,\ldots,n}\bigg|\frac{4\beta_i\rho_1^2}{\omega_1}+1,536\rho_1^2\sum_{k,l=1}^n\gamma_{i,k}^l\frac{\omega_2^4}{\omega_1^5}
+2\omega_1+\left(2\omega_1+2\omega_1^2
+8\rho_1^2\right)(d-1)+\frac{8\alpha\rho_1^2}{\omega_1}\bigg|.
\end{align*}
If $(\mu,\nu)$ lie in $B(L)$, then the estimate
\begin{align}\label{bd_F_Theta}
\sup_{r\in[0,\sigma]}&\big|\bar F\big(\bar h(r)+\nu(r),\bar
f(r)+\mu(r),\bar f'(r)+\mu'(r),\phi(r),\phi'(r)\big)\big|_{\mathbb
R^n}\le\Theta
\end{align}
holds true. Moreover, in this case, we have
\begin{align}\label{bnds_Hartman}
|u_{\mu,\nu}(r)|_{\mathbb
R^n}\le\frac{\sigma^2}8\Theta,\qquad|u_{\mu,\nu}'(r)|_{\mathbb
R^n}\le\frac\sigma2\Theta,\qquad r\in[0,\sigma].
\end{align}
\end{lemma}
\begin{proof}
Inequality~\eqref{bd_F_Theta} is a straightforward consequence of the
definition of $\bar F$. To obtain~\eqref{bnds_Hartman}, it suffices to write down an explicit formula for $u_{\mu,\nu}$ and perform elementary estimation of its terms (for the second part of~\eqref{bnds_Hartman}, one needs to differentiate before estimating). We refer to~\cite[Section~XII.4]{PH64} for the details of this argument.
\end{proof}

From now on and until the end of Section~\ref{subsec_pf_lemma}, we
assume
\begin{align}\label{bd_sigma_epsilon0}
\sigma\le\min\left\{\sigma_1,\sqrt{\frac{\omega_1}\Theta}\,,\frac{\omega_1}{2d\Theta}\right\}.
\end{align}
Given $(\mu,\nu)\in B(L)$, let us introduce
$v_{\mu,\nu}:[0,\sigma]\to\mathbb R$ through the formula
\begin{align}\label{def_zeta_mu,nu}
v_{\mu,\nu}(r)=&-\bar h(0)+H\big(\bar f(0)+u_{\mu,\nu}(0),\bar
f'(0)+u_{\mu,\nu}'(0),\phi(0)\big) \notag \\
&+\int_0^r\big(-\bar h'(s)+K\big(\bar h(s)+\nu(s),\bar
f(s)+u_{\mu,\nu}(s),\bar
f'(s)+u_{\mu,\nu}'(s),\phi'(s)\big)\big)\,ds,\qquad
r\in[0,\sigma].
\end{align}
Lemma~\ref{lemma_bd_F_Theta} and
inequality~\eqref{bd_sigma_epsilon0} imply the estimates
\begin{align*}
\sup_{r\in[0,\sigma]}|u_{\mu,\nu}(r)|_{\mathbb
R_n}\le\frac{\omega_1}2\,,\qquad
\sup_{r\in[0,\sigma]}|u_{\mu,\nu}'(r)|_{\mathbb R_n}\le\frac{\omega_1}{4d}\,,
\end{align*}
which ensure that the right-hand side of~\eqref{def_zeta_mu,nu} is
well-defined (indeed, the expression $H\big(\bar f(0),y,\phi(0)\big)$ is
well-defined and positive whenever $|y|_{\mathbb R^n}\le\frac{\omega_1}{2d}$). We now set $\mathcal
C(\mu,\nu)=(u_{\mu,\nu},v_{\mu,\nu})$. Our next goal is to show that, when $\sigma$ is sufficiently small and the radius
$L$ is appropriately chosen, the map $\mathcal C$ satisfies the conditions of the Schauder theorem.

Suppose $\theta>0$ is
a constant obeying the inequalities
\begin{align}\label{H_Lip_const}
|H(x,y,z\big)-H(x,\hat
y,z\big)|&\le\theta|y-\hat y|_{\mathbb R^n}, \notag \\
|H(x,y,z\big)-H(x,\hat
y,z\big)|&\le\theta\sum_{k,l=1}^n|y_ky_l-\hat y_k\hat
y_l|, \notag \\
x&\in[\omega_1,\omega_2]^n, \notag \\ y&=(y_1,\ldots,y_n)\in\bigg[-\frac{\omega_1}{2d},\frac{\omega_1}{2d}\bigg]^n,~\hat
y=(\hat y_1,\ldots,\hat y_n)\in\bigg[-\frac{\omega_1}{2d},\frac{\omega_1}{2d}\bigg]^n,\notag
\\ z&\in\bigg\{(z_1,\ldots,z_n)\in[0,\alpha]^n\,\bigg|\,\sum_{k=1}^nd_kz_k\ge\rho_0(\omega_1,\omega_2)\bigg\},
\end{align}
and the inequality
\begin{align}\label{K_Lip_const}
|K(p,x,y,w)&-K(\hat p,\hat x,\hat y,w)|\le\theta(|p-\hat
p|+|x-\hat x|_{\mathbb R^n}+|y-\hat y|_{\mathbb R^n}), \notag \\
p,\hat p&\in\bigg[\frac1{2\rho_1},2\rho_1\bigg],~x,\hat
x\in\bigg[\frac{\omega_1}2,2\omega_2\bigg]^n,~y,\hat
y\in\bigg[-\frac{\omega_1}{2d},\frac{\omega_1}{2d}\bigg]^n,~w\in[-1,1]^n.
\end{align}
It is obvious that such a $\theta$ exists. We define
\begin{align*} \Sigma&=\Theta+\theta n^2(2\Theta+\Theta^2+\omega_1), \\
\sigma_0(\alpha,\omega_1,\omega_2)&=
\min\left\{\sigma_1,\sqrt{\frac{\omega_1}\Theta}\,,\frac{\omega_1}{2d\Theta}\,,\frac{\omega_1}{2\Sigma}\,,\frac1{2\rho_1\Sigma}\right\}.
\end{align*}
Note that $\sigma_0$ is the function whose existence (along with $\rho_0$) Lemma~\ref{lemma_exist_hf} asserts. Let us also set $L_0=\sigma^2\Sigma$. From now on, we will assume the second inequality
in~\eqref{hyp_thm_bounds} holds, i.e., $\sigma<\sigma_0(\alpha,\omega_1,\omega_2)$. This implies, in particular, that
$L_0$ cannot exceed
$\frac\sigma2\min\big\{\omega_1,\frac1{\rho_1}\big\}$.

\begin{lemma}
The image $\mathcal CB(L_0)$ is contained in $B(L_0)$.
\end{lemma}

\begin{proof}
Take a pair $(\mu,\nu)$ from $B(L_0)$. Our goal is to show that $\mathcal
C(\mu,\nu)$ lies in $B(L_0)$. Lemma~\ref{lemma_bd_F_Theta} yields the
estimate
\begin{align*}
|(u_{\mu,\nu},v_{\mu,\nu})|_{\mathcal
B}\le\sigma^2\Theta+\sup_{r\in[0,\sigma]}|v_{\mu,\nu}(r)|.
\end{align*}
Remembering the second formula in~\eqref{hyp_thm_deriv}, we also find
\begin{align*}
|v_{\mu,\nu}(r)|&\le\big|-\bar h(0)+H\big(\bar
f(0)+u_{\mu,\nu}(0),\bar
f'(0)+u_{\mu,\nu}'(0),\phi(0)\big)\big| \\
&\hphantom{=}~+\sigma\sup_{s\in[0,r]}\big|-\bar h'(s)+K\big(\bar
h(s)+\nu(s),\bar f(s)+u_{\mu,\nu}(s),\bar
f'(s)+u_{\mu,\nu}'(s),\phi'(s)\big)\big| \\
&=\big|H\big(\bar f(0),\bar f'(0)+u_{\mu,\nu}'(0),\phi(0)\big)-H\big(\bar f(0),\bar f'(0),\phi(0)\big)\big| \\
&\hphantom{=}~+\sigma\sup_{s\in[0,r]}\big|K\big(\bar
h(s)+\nu(s),\bar f(s)+u_{\mu,\nu}(s),\bar f'(s)+u_{\mu,\nu}'(s),\phi'(s)\big)-K\big(\bar
h(s),\bar f(s),\bar f'(s),\phi'(s)\big)\big|
\\ &\le\theta\sum_{k,l=1}^n(\sigma|(u_{\mu,\nu})_k'(0)|+\sigma|(u_{\mu,\nu})_l'(0)|+|(u_{\mu,\nu})_k'(0)(u_{\mu,\nu})_l'(0)|)
\\ &\hphantom{=}~+\sigma\theta\sup_{s\in[0,r]}(|\nu(s)|+|u_{\mu,\nu}(s)|_{\mathbb R^n}+|u_{\mu,\nu}'(s)|_{\mathbb R^n})
\\ &\le\sigma^2\theta n^2(2\Theta+\Theta^2+\omega_1),\qquad
r\in[0,\sigma],
\end{align*}
where $(u_{\mu,\nu})_k$ and $(u_{\mu,\nu})_l$ are the $k$th and
the $l$th components of $u_{\mu,\nu}$. Consequently, it must be
the case that
\begin{align*}
|(u_{\mu,\nu},v_{\mu,\nu})|_{\mathcal B}
\le\sigma^2(\Theta+\theta n^2(2\Theta+\Theta^2+\omega_1))=
\sigma^2\Sigma=L_0.
\end{align*}
\end{proof}

\begin{lemma}
The map $\mathcal C:B(L_0)\to\mathcal B$ is continuous.
\end{lemma}

\begin{proof}
Without loss of generality, assume the constant $\theta$ fixed above satisfies
\begin{align}\label{F_Lip_const}
\big|\bar F(p,x,y,z,w)&-\bar F(\hat p,\hat x,\hat
y,z,w)\big|_{\mathbb R^n}\le\theta(|p-\hat p|+|x-\hat x|_{\mathbb
R^n}+|y-\hat y|_{\mathbb R^n}), \notag
\\
p,\hat p&\in\left[\frac1{2\rho_1},2\rho_1\right],~x,\hat
x\in\left[\frac{\omega_1}2,2\omega_2\right]^n,~y,\hat
y\in\left[-\frac{\omega_1}2,\frac{\omega_1}2\right]^n,~z\in[0,\alpha]^n,~w\in[-1,1]^n.
\end{align}
Suppose the pairs $(\mu_1,\nu_1)$ and $(\mu_2,\nu_2)$ lie in $B(L_0)$.
The first formula in~\eqref{eq_centered_xi}, along with inequalities~\eqref{bnds_Hartman} and~\eqref{F_Lip_const},
imply
\begin{align*}
\sup_{r\in[0,\sigma]}|u_{\mu_1,\nu_1}(r)-u_{\mu_2,\nu_2}(r)|_{\mathbb
R^n}&\le\frac{\sigma\theta}8|(\mu_1,\nu_1)-(\mu_2,\nu_2)|_{\mathcal
B},
\\
\sup_{r\in[0,\sigma]}|u_{\mu_1,\nu_1}'(r)-u_{\mu_2,\nu_2}'(r)|_{\mathbb
R^n}&\le \frac{\theta}2|(\mu_1,\nu_1)-(\mu_2,\nu_2)|_{\mathcal B}.
\end{align*}
Using~\eqref{def_zeta_mu,nu},~\eqref{H_Lip_const},
and~\eqref{K_Lip_const}, we also find
\begin{align*}
\sup_{r\in[0,\sigma]}|v_{\mu_1,\nu_1}(r)-v_{\mu_2,\nu_2}(r)|&\le\theta|u_{\mu_1,\nu_1}'(0)-u_{\mu_2,\nu_2}'(0)|_{\mathbb
R^n}+\theta\int_0^\sigma|\nu_1(s)-\nu_2(s)|\,ds
\\ &\hphantom{=}~+\theta\int_0^\sigma(|u_{\mu_1,\nu_1}(s)-u_{\mu_2,\nu_2}(s)|_{\mathbb
R^n}+|u_{\mu_1,\nu_1}'(s)-u_{\mu_2,\nu_2}'(s)|_{\mathbb
R^n})\,ds \\ &\le
\left(\frac{3\theta^2}2+\theta\right)|(\mu_1,\nu_1)-(\mu_2,\nu_2)|_{\mathcal
B}.
\end{align*}
Consequently, it must be the case that
\begin{align*}
|\mathcal C(\mu_1,\nu_1)-\mathcal C(\mu_2,\nu_2)|_{\mathcal B}\le
\left(\frac{3\theta^2}2+2\theta\right)|(\mu_1,\nu_1)-(\mu_2,\nu_2)|_{\mathcal
B},
\end{align*}
which tells us $\mathcal C$ is continuous.
\end{proof}

\begin{lemma}
The closure of the set $\mathcal CB(L_0)$ in $\mathcal B$ is a compact
subset of $\mathcal B$.
\end{lemma}

\begin{proof}
Suppose $((\mu_j,\nu_j))_{j=1}^\infty$ are pairs from $B(L_0)$. It
suffices to prove that the sequence
$((u_{\mu_j,\nu_j},v_{\mu_j,\nu_j}))_{j=1}^\infty$ has a
convergent subsequence. The mean value theorem and
Lemma~\ref{lemma_bd_F_Theta} yield the estimates
\begin{align*} |u_{\mu_j,\nu_j}(r_1)-u_{\mu_j,\nu_j}(r_2)|_{\mathbb
R^n}&\le\sup_{r\in[0,\sigma]}|u_{\mu_j,\nu_j}'(r)|_{\mathbb
R^n}|r_1-r_2|\le\frac\sigma2\Theta|r_1-r_2|, \\
|u_{\mu_j,\nu_j}'(r_1)-u_{\mu_j,\nu_j}'(r_2)|_{\mathbb
R^n}&\le\sup_{r\in[0,\sigma]}|u_{\mu_j,\nu_j}''(r)|_{\mathbb
R^n}|r_1-r_2|\le\Theta|r_1-r_2|,\qquad j=1,2,\ldots,
\end{align*}
for $r_1,r_2\in[0,\sigma]$. Recalling formulas~\eqref{eq_bar_h}
and~\eqref{K_Lip_const}, we also obtain
\begin{align*}
|v_{\mu_j,\nu_j}(r_1)&-v_{\mu_j,\nu_j}(r_2)|\le\sup_{r\in[0,\sigma]}|v_{\mu_j,\nu_j}'(r)||r_1-r_2|
\\ &\le\sup_{r\in[0,\sigma]}\big|-\bar h'(r)+K\big(\bar
h(r)+\nu_j(r),\bar f(r)+u_{\mu_j,\nu_j}(r),\bar
f'(r)+u_{\mu_j,\nu_j}'(r),\phi'(r)\big)\big||r_1-r_2|
\\ &\le\theta\sup_{r\in[0,\sigma]}(|\nu_j(r)|+|u_{\mu_j,\nu_j}(r)|_{\mathbb R^n}+|u_{\mu_j,\nu_j}'(r)|_{\mathbb R^n})|r_1-r_2| \\
&\le\theta(\sigma^2\Sigma+\sigma\Theta)|r_1-r_2|,\qquad
j=1,2,\ldots,~r_1,r_2\in[0,\sigma].
\end{align*}
It follows that the sequences $(u_{\mu_j,\nu_j})_{j=1}^\infty$,
$(u_{\mu_j,\nu_j}')_{j=1}^\infty$, and
$(v_{\mu_j,\nu_j})_{j=1}^\infty$ are equicontinuous.
Furthermore, because $\mathcal CB(L_0)$ is a subset of $B(L_0)$, they are
uniformly bounded. These facts, along with the Arzel\`a-Ascoli
theorem, imply that
$((u_{\mu_j,\nu_j},v_{\mu_j,\nu_j}))_{j=1}^\infty$ must have a
convergent subsequence.
\end{proof}

According to the lemmas above, the map $\mathcal C:B(L_0)\to\mathcal B$ is continuous,
and its image is a precompact subset of $B(L_0)$. Consequently, the Schauder theorem
(see, e.g.,~\cite[Chapter~XII, Corollary~0.1]{PH64}) allows us to conclude that there exists a pair $(u,v)\in B(L_0)$ satisfying the equality $\mathcal
C(u,v)=(u,v)$. It is easy to understand that $u$ and $v$
obey~\eqref{sys_centered} and~\eqref{BC_centered}. A simple
bootstrapping argument demonstrates that $u$ and $v$ are smooth. We
define $f=\bar f+u$ and $h=\bar h+v$. Clearly, these functions take
values in $(0,\infty)^n$ and $(0,\infty)$, respectively, and
solve~\eqref{sys_to_solve} and~\eqref{eq_Bianchi_K} under the
conditions~\eqref{BC_on_f} and~\eqref{BC_on_h}. Thus,
Lemma~\ref{lemma_exist_hf} is established.

\subsection{Proof of Theorems~\ref{thm_loc_ex_int}
and~\ref{thm_loc_ex_bdy}}\label{subsec_thm_loc_pf}

Given $\tau\in[0,\sigma]$ and $\kappa>0$, set
\begin{align*}
J_\kappa^\tau=(\tau-\kappa,\tau+\kappa)\cap[0,\sigma], \qquad \mathcal X_\kappa^\tau=J_\kappa^\tau\times \mathcal G/\mathcal H.
\end{align*}
Obviously, $\mathcal X_\kappa^\tau$ is a neighbourhood of
$\Gamma^\tau$ in $M$. Assume $G^\tau$ is a
Riemannian metric on $M$ and $G^\tau_{\Gamma^\tau}$ is the metric on
$\Gamma^\tau$ induced by $G^\tau$. Let $\II(G^\tau)$ be the second fundamental form
of $\Gamma^\tau$ in $M$ with respect to $G^\tau$ and to the unit normal whose scalar
product with $\big(\frac\partial{\partial
r},0\big)$ is less than~0.

For each $\tau\in[0,\sigma]$, consider a symmetric positive-definite
$\mathcal G$-invariant $(0,2)$-tensor field $R^\tau$ on $\Gamma^\tau$. In
order to keep our notation consistent, we assume $R^0$ and $R^\sigma$ are
the restrictions of $R$ to $\{0\}\times \mathcal G/\mathcal H$ and $\{\sigma\}\times \mathcal G/\mathcal H$. It is evident
that $R^\tau$ is fully determined by how
it acts on $\mathfrak p$. There exist
numbers $a_{\tau,1},\ldots,a_{\tau,n}>0$ satisfying
\begin{align*}
R^\tau(X,Y)&=a_{\tau,1}^2\,Q\big(\pr_{\mathfrak
p_1}X,\pr_{\mathfrak
p_1}Y\big)+\cdots+a_{\tau,n}^2\,Q\big(\pr_{\mathfrak
p_n}X,\pr_{\mathfrak p_n}Y\big),\qquad X,Y\in\mathfrak p.
\end{align*}
Let us also fix, for every $\tau\in[0,\sigma]$, a symmetric $\mathcal G$-invariant
tensor field $S^\tau$ on $\Gamma^\tau$. There are
$\delta_{\tau,1},\ldots,\delta_{\tau,n}\in\mathbb R$ such that
\begin{align*}
S^\tau(X,Y)&=\delta_{\tau,1}\,Q\big(\pr_{\mathfrak
p_1}X,\pr_{\mathfrak
p_1}Y\big)+\cdots+\delta_{\tau,n}\,Q\big(\pr_{\mathfrak
p_n}X,\pr_{\mathfrak p_n}Y\big),\qquad X,Y\in\mathfrak p.
\end{align*}

Proposition~\ref{prop_local_ex}, which we are about to state,
underlies Theorems~\ref{thm_loc_ex_int} and~\ref{thm_loc_ex_bdy}.
The author's paper~\cite{AP11} contains similar results, though established in a different setting. Another closely related theorem was obtained in~\cite{MAMH08}. The methods we use in the present paper are different from those of~\cite{AP11}, as we explain in Remark~\ref{rem_methods} below.

\begin{proposition}\label{prop_local_ex}
Suppose $\tau\in[0,\sigma]$. The following two statements are equivalent:
\begin{enumerate}
\item\label{stt1_prop_loc}
For some $\kappa>0$, there exists a $\mathcal G$-invariant Riemannian
metric $G^\tau$ on $M$ such that
$\Ric(G^\tau)=T$ on $\mathcal X_\kappa^\tau$, $G_{\Gamma^\tau}^\tau=R^\tau$, and $\II(G^\tau)=S^\tau$.
\item\label{stt2_prop_loc}
The inequality
\begin{align}\label{nec_suf_quant}
\sum_{k=1}^nd_k\Bigg(\frac{\beta_k}{2a_{\tau,k}^2}
+\sum_{l,m=1}^n\gamma_{k,l}^m\frac{a_{\tau,k}^4-2a_{\tau,l}^4}{4a_{\tau,k}^2a_{\tau,l}^2a_{\tau,m}^2}
-\sum_{l=1}^nd_l\frac{\delta_{\tau,k}\delta_{\tau,l}}{a_{\tau,k}^2a_{\tau,l}^2}+\frac{\delta_{\tau,k}^2}{a_{\tau,k}^4}
-\frac1{a_{\tau,k}^2}\phi_k(\tau)\Bigg)<0
\end{align}
is satisfied.
\end{enumerate}
If these statements hold and $\check{G}^\tau$ is a
$\mathcal G$-invariant metric on $M$ such that
$\Ric(\check{G}^\tau)=T$ on $\mathcal X_\kappa^\tau$,
$\check{G}^\tau_{\Gamma^\tau}=R^\tau$, and
$\II(\check{G}^\tau)=S^\tau$, then
$\check{G}^\tau$ must coincide with $G^\tau$ on $\mathcal X_\kappa^\tau$.
\end{proposition}

\begin{proof}
Suppose there exist $\kappa>0$ and a $\mathcal G$-invariant Riemannian
metric $G^\tau$ on $M$ such that
$\Ric(G^\tau)=T$ on $\mathcal X_\kappa^\tau$, $G_{\Gamma^\tau}^\tau=R^\tau$, and $\II(G^\tau)=S^\tau$. Employing Lemma~\ref{lemma_Ricci_comp} and the fact
that $T$ is positive-definite, one can show that $G^\tau$
satisfies the formula
\begin{align}\label{cal_Ht_form1}
G^\tau=h_\tau^2(r)\,dr\otimes dr+G_r^\tau,\qquad r\in
J_\kappa^\tau.
\end{align}
Here, $h_\tau$ is a smooth function acting from $J_\kappa^\tau$ to
$(0,\infty)$. The tensor field $G_r^\tau$ is a
$\mathcal G$-invariant Riemannian metric on $\mathcal G/\mathcal H$. It is clear that
\begin{align}\label{cal_Ht_form2}
G_r^\tau(X,Y)=f_{\tau,1}^2(r)\,Q\big(\pr_{\mathfrak
p_1}X,\pr_{\mathfrak
p_1}Y\big)+\cdots+f_{\tau,n}^2(r)\,Q\big(\pr_{\mathfrak
p_n}X,\pr_{\mathfrak p_n}Y\big),\qquad X,Y\in\mathfrak p,
\end{align}
for some smooth functions $f_{\tau,1},\ldots,f_{\tau,n}$ from
$J_\kappa^\tau$ to $(0,\infty)$. The equality $\Ric(G^\tau)=T$ and Lemma~\ref{lemma_Ricci_comp} imply
\begin{align}\label{pf_loc_aux1}
H_1(f_\tau(r),f_\tau'(r))=h_\tau^2(r)H_2(f_\tau(r),\phi(r)),\qquad
r\in J_\kappa^\tau.
\end{align}
The notation $f_\tau$ here stands for
$(f_{\tau,1},\ldots,f_{\tau,n})$. Because $G_{\Gamma^\tau}^\tau=R^\tau$ and $\II(G^\tau)=S^\tau$, we also have
\begin{align*}
f_\tau(\tau)=a_\tau,\qquad
f_\tau'(\tau)=-h_\tau(\tau)\delta_\tau^a,
\end{align*}
where $a_\tau=(a_{\tau,1},\ldots,a_{\tau,n})$ and
$\delta_\tau^a=\big(\frac{\delta_{\tau,1}}{a_{\tau,1}},\ldots,\frac{\delta_{\tau,n}}{a_{\tau,n}})$.
Keeping these two formulas in mind and using~\eqref{pf_loc_aux1}, we
easily calculate that the quantity in the left-hand side
of~\eqref{nec_suf_quant} is equal to
$-\frac1{h_\tau^2(\tau)}$. This quantity must, therefore, be
negative.

Assume now that~\eqref{nec_suf_quant} holds. Let us prove the
existence of $\kappa>0$ and a metric $G^\tau$ on
$M$ such that $\Ric(G^\tau)=T$ on
$\mathcal X_\kappa^\tau$, $G_{\Gamma^\tau}^\tau=R^\tau$,
and $\II(G^\tau)=S^\tau$. Consider the system of ordinary differential
equations
\begin{align}\label{loc_ode}
f_\tau''(r)&=\bar F(h_\tau(r),f_\tau(r),f_\tau'(r),\phi(r),\phi'(r)), \notag \\
h_\tau'(r)&=K(h_\tau(r),f_\tau(r),f_\tau'(r),\phi'(r)),
\end{align}
for the unknown functions $f_\tau$ and $h_\tau$. We supplement this
system with the conditions
\begin{align}\label{loc_ini_cond}
f_\tau(\tau)&=a_\tau, \notag \\
f_\tau'(\tau)&=-(H_2(a_\tau,\phi(\tau))+(1-H_1(a_\tau,\delta_\tau^a)))^{-\frac12}\delta_\tau^a, \notag \\
h_\tau(\tau)&=(H_2(a_\tau,\phi(\tau))+(1-H_1(a_\tau,\delta_\tau^a)))^{-\frac12}.
\end{align}
Note that, thanks to~\eqref{nec_suf_quant}, the right-hand sides of
the last two formulas are well-defined. The standard theory of
ordinary differential equations tells us that problem~\eqref{loc_ode}--\eqref{loc_ini_cond}
has a solution. To be more precise, for some number
$\kappa>0$, there exist smooth functions
$f_\tau:J_\kappa^\tau\to(0,\infty)^n$ and
$h_\tau:J_\kappa^\tau\to(0,\infty)$ solving~\eqref{loc_ode} on
$J_\kappa^\tau$ and satisfying~\eqref{loc_ini_cond}. With these
functions at hand, we define a $\mathcal G$-invariant Riemannian metric
$G^\tau$ on $\mathcal X_\kappa^\tau$ by
formulas~\eqref{cal_Ht_form1} and~\eqref{cal_Ht_form2}. We extend it to all of $M$ arbitrarily. It
follows from~\eqref{loc_ode} that
\begin{align*}
\Ric(G^\tau)=\hat\sigma(r)\,dr\otimes dr+T_r,\qquad
r\in J_\kappa^\tau,
\end{align*}
for some $\hat\sigma:J_\kappa^\tau\to\mathbb R$. Employing
Lemma~\ref{lemma_Bianchi} and arguing as in
Section~\ref{subsec_less_lemma}, one demonstrates that $\hat\sigma$
must be identically equal to~1 on $J_\kappa^\tau$. This means
$\Ric(G^\tau)=T$ on $\mathcal X_\kappa^\tau$.
Conditions~\eqref{loc_ini_cond} imply that $G_{\Gamma^\tau}^\tau=R^\tau$ and $\II(G^\tau)=S^\tau$.

Suppose now that statements~\ref{stt1_prop_loc}
and~\ref{stt2_prop_loc} in Proposition~\ref{prop_local_ex} hold
true. We may assume the metric $G^\tau$
satisfies~\eqref{cal_Ht_form1} and~\eqref{cal_Ht_form2}. Then the
functions $f_\tau$ and $h_\tau$
solve~\eqref{loc_ode}--\eqref{loc_ini_cond} on $J_\kappa^\tau$.
Consider a $\mathcal G$-invariant Riemannian metric $\check{G}^\tau$
on $M$ such that $\Ric(\check{G}^\tau)=T$ on $\mathcal X_\kappa^\tau$, $\check{G}_{\Gamma^\tau}^\tau=R^\tau$, and
$\II(\check{G}^\tau)=S^\tau$. Our objective
is to show that $\check{G}^\tau$ coincides with $G^\tau$ on $\mathcal X_\kappa^\tau$. By analogy with~\eqref{cal_Ht_form1}, we write
\begin{align*}
\check{G}^\tau=\check h_\tau^2(r)\,dr\otimes
dr+\check{G}_r^\tau,\qquad r\in J_\kappa^\tau.
\end{align*}
In the right-hand side, $\check h_\tau:J_\kappa^\tau\to(0,\infty)$
is a smooth function. The tensor field $\check{G}_r^\tau$
is a $\mathcal G$-invariant Riemannian metric on $\mathcal G/\mathcal H$. There are smooth
functions $\check f_{\tau,1},\ldots,\check f_{\tau,n}$ from
$J_\kappa^\tau$ to $(0,\infty)$ such that
\begin{align*}
\check{G}_r^\tau(X,Y)=\check
f_{\tau,1}^2(r)\,Q\big(\pr_{\mathfrak p_1}X,\pr_{\mathfrak
p_1}Y\big)+\cdots+\check f_{\tau,n}^2(r)\,Q\big(\pr_{\mathfrak
p_n}X,\pr_{\mathfrak p_n}Y\big),\qquad X,Y\in\mathfrak p.
\end{align*}
It will be convenient for us to denote $\check f_\tau=\big(\check
f_{\tau,1},\ldots,\check f_{\tau,n}\big)$. Because
$\Ric(\check{G}^\tau)=T$, $\check{G}_{\Gamma^\tau}^\tau=R^\tau$, and
$\II(\check{G}^\tau)=S^\tau$,
formulas~\eqref{loc_ode}--\eqref{loc_ini_cond} would still hold on
$J_\kappa^\tau$ if we substituted $\check f_\tau$, $\check f_\tau'$,
$\check f_\tau''$, $\check h_\tau$, and $\check h_\tau'$ in them for
$f_\tau$, $f_\tau'$, $f_\tau''$, $h_\tau$, and $h_\tau'$. The
standard theory of ordinary differential equations then implies that
$\check f_\tau=f_\tau$ and $\check h_\tau=h_\tau$ on
$J_\kappa^\tau$. Consequently, $\check{G}^\tau$ coincides
with $G^\tau$ on $\mathcal X_\kappa^\tau$. Thus, the proof is complete.
\end{proof}

\begin{remark}\label{rem_methods} One may establish
Proposition~\ref{prop_local_ex} by adapting the methods employed in the
paper~\cite{AP11}. The main idea behind those methods is to modify the right-hand side of~\eqref{intro_PRC} by a $G$-dependent diffeomorphism making the equation more easily solvable. Such an approach relies on the work of DeTurck (see~\cite[Chapter~5]{AB87} for an overview) and is similar in spirit to the DeTurck trick for the Ricci flow.
Conversely, it seems possible to prove the existence and uniqueness results in~\cite{AP11} with the techniques employed above.
\end{remark}

\begin{remark}
Suppose we are in the situation described in Remark~\ref{rem_abelian}. Thus, $\mathcal G/\mathcal H$ is an abelian Lie group, and Hypothesis~\ref{assum_decomp_p} fails to hold. In this case,
statement~\ref{stt2_prop_loc} of Proposition~\ref{prop_local_ex}
implies statement~\ref{stt1_prop_loc}, but
establishing the converse implication may be problematic. Roughly
speaking, this is because, when Hypothesis~\ref{assum_decomp_p} does
not hold, the metric $G^\tau$ need not be diagonal with
respect to~\eqref{p_decomp}. For the same reason, proving the
assertion about
$\check{G}^\tau$ may be troublesome with our methods.
\end{remark}

\begin{remark}
Statement~\ref{stt1_prop_loc} of Proposition~\ref{prop_local_ex} is equivalent to
statement~\ref{stt2_prop_loc} even if $T$ is not positive-definite. Yet our methods do not yield the
assertion about $\check{G}^\tau$ in this case. Recall that, if $T$ is not positive-definite, we need to assume the existence of a diffeomorphism $\Psi$ satisfying~\eqref{dif_def}, \eqref{T_form1} and~\eqref{T_form2}.
\end{remark}

Theorems~\ref{thm_loc_ex_int}
and~\ref{thm_loc_ex_bdy} follow from Proposition~\ref{prop_local_ex} by choosing
$R^\tau$ and $S^\tau$ in such a way that
\begin{align*}
R^\tau(X,Y)&=\bigg(\frac{\sigma-\tau}{\sigma}a_1+\frac{\tau}{\sigma}
b_1\bigg)\,Q\big(\pr_{\mathfrak p_1}X,\pr_{\mathfrak
p_1}Y\big)\\ &\hphantom{=}~+\cdots +\bigg(\frac{1-\tau}{\sigma}a_n+\frac{\tau}{\sigma} b_n\bigg)\,Q\big(\pr_{\mathfrak
p_n}X,\pr_{\mathfrak p_n}Y\big),\\
S^\tau(X,Y)&=\beta\,Q(X,Y),\qquad
X,Y\in\mathfrak p,
\end{align*}
for a sufficiently large $\beta>0$.


\begin{thebibliography}{99}

\bibitem{MA08}
M.T. Anderson, On boundary value problems for Einstein metrics,
Geom. Topol.~12 (2008) 2009--2045.

\bibitem{MA12}
M.T. Anderson, Boundary value problems for metrics on 3-manifolds,
in: X.~Dai and X.~Rong (Eds.), Metric and differential geometry, in
honor of J.~Cheeger, Birkh\"auser Verlag, Basel, 2012, pp. 3--17.

\bibitem{MAMH08}
M.T. Anderson, M. Herzlich, Unique continuation results for Ricci
curvature and applications, J.~Geom. Phys. 58 (2008) 179--207;
erratum in J.~Geom. Phys. 60 (2010) 1062--1067.

\bibitem{TA98}
T. Aubin, Some nonlinear problems in Riemannian geometry,
Springer-Verlag, Berlin, 1998.

\bibitem{MBXCAP10}
M.~Bailesteanu, X.~Cao, A.~Pulemotov, Gradient estimates for the
heat equation under the Ricci flow, J.~Funct. Anal.~258 (2010)
3517--3542.

\bibitem{LBB82} L. B\'erard Bergery, Sur de nouvelles vari´et´es riemanniennes d'Einstein, Institut \'Elie Cartan~6 (1982) 1--60.

\bibitem{AB87}
A. Besse, Einstein manifolds, Springer-Verlag, Berlin, 1987.

\bibitem{CB98}
C. B\"ohm, Inhomogeneous Einstein metrics on low-dimensional spheres
and other low-dimensional spaces, Invent. Math.~134 (1998) 145--176.

\bibitem{JCDDT94}
J. Cao, D.M. DeTurck, The Ricci curvature equation with rotational
symmetry, Amer. J.~Math.~116 (1994) 219--241.

\bibitem{JC09}
J.C. Cortissoz, Three-manifolds of positive curvature and convex
weakly umbilic boundary, Geom. Dedicata~138 (2009) 83--98.

\bibitem{ADMW99}
A.S. Dancer, M.Y. Wang, Integrable cases of the Einstein equations,
Comm. Math. Phys.~208 (1999) 225--243.

\bibitem{ADMW00}
A.S. Dancer, M.Y. Wang, The cohomogeneity one Einstein equations
from the Hamiltonian viewpoint, J.~reine angew. Math.~524 (2000)
97--128.

\bibitem{ADMW11}
A.S. Dancer, M.Y. Wang, On Ricci solitons of cohomogeneity one, Ann.
Glob. Anal. Geom.~39 (2011) 259--292.

\bibitem{JDDKAY10}
J. DeBlois, D. Knopf, A. Young, Cross curvature flow on a negatively
curved solid torus, Algebr. Geom. Topol.~10 (2010) 343--372.

\bibitem{PhD03}
Ph. Delano\"e, Local solvability of elliptic, and curvature,
equations on compact manifolds, J.~reine angew. Math.~558 (2003)
23--45.

\bibitem{ED02}
E. Delay, Studies of some curvature operators in a neighborhood of
an asymptotically hyperbolic Einstein manifold, Adv. Math.~168
(2002) 213--224.

\bibitem{PG12}
P. Gianniotis, The Ricci flow on manifolds with boundary,
arXiv:1210.0813 [math.DG].

\bibitem{KGWZ02}
K. Grove, W. Ziller, Cohomogeneity one manifolds with positive Ricci
curvature, Invent. Math.~149 (2002) 619--646.

\bibitem{RH84}
R.S. Hamilton, The Ricci curvature equation, in: S.-S.~Chern (Ed.),
Seminar on nonlinear partial differential equations,
Springer-Verlag, New York, 1984, pp.~47--72.

\bibitem{PH64}
P. Hartman, Ordinary differential equations, John Wiley \& Sons, New
York, 1964.

\bibitem{CH10}
C.A. Hoelscher, Classification of cohomogeneity one manifolds in low
dimensions, Pacific J.~Math.~246 (2010) 129--185.

\bibitem{RPKT09}
R. Pina, K. Tenenblat, On solutions of the Ricci curvature equation
and the Einstein equation, Israel J.~Math.~171 (2009) 61--76.

\bibitem{AP10}
A. Pulemotov, Quasilinear parabolic equations and the Ricci flow on
manifolds with boundary, J. reine angew. Math.~683 (2013) 97--118.

\bibitem{AP11}
A. Pulemotov, Metrics with prescribed Ricci curvature near the
boundary of a manifold, Math. Ann.~357 (2013) 969--986. 

\bibitem{AP14}
A. Pulemotov, The Ricci flow on domains in cohomogeneity one manifolds, submitted, arXiv:1410.7505 [math.AP].

\bibitem{AP15}
A. Pulemotov, Metrics with prescribed Ricci curvature on homogeneous spaces, submitted, arXiv:1504.01498 [math.DG].

\bibitem{YR08a}
Y.A. Rubinstein, Some discretizations of geometric evolution
equations and the Ricci iteration on the space of K\"{a}hler
metrics, Adv. Math.~218 (2008) 1526--1565.

\bibitem{YR08b}
Y.A. Rubinstein, Some discretizations of geometric evolution
equations and the Ricci iteration on the space of K\"{a}hler
metrics, Ph.D. thesis, Massachusetts Institute of Technology, 2008.

\bibitem{YS96}
Y. Shen, On Ricci deformation of a Riemannian metric on manifold
with boundary, Pacific J.~Math.~173 (1996) 203--221.

\end{thebibliography}
\end{document}